\documentclass[12pt]{article}

\usepackage[square,sort,comma,numbers]{natbib}

\usepackage{amsthm}
\usepackage[english]{babel}
\usepackage[utf8]{inputenc}
\usepackage{amsmath}
\usepackage{amsthm}
\usepackage{graphicx}
\usepackage[colorinlistoftodos]{todonotes}
\usepackage{wasysym}
\usepackage{amsfonts}
\usepackage{ amssymb }
\usepackage{tabu}
\usepackage{bbm}
\usepackage{mathtools}
\usepackage{esvect}
\usepackage{amsmath,graphicx}
\usepackage{amsmath,lipsum}
\usepackage{enumitem}
\usepackage{physics}
\usepackage{tikz}
\usetikzlibrary{trees}
\usepackage{centernot}
\usepackage{extsizes}
\usepackage{authblk}

\theoremstyle{plain}
\usepackage{geometry}
\author{Nina Zubrilina}

\affil{Stanford University Department of Mathematics\\ Email: nizubrilina@gmail.com}

\date{\today}

\begin{document}

\newtheorem{theorem}{Theorem}[section]
\newtheorem{observation}[theorem]{Observation}
\newtheorem{corollary}[theorem]{Corollary}
\newtheorem{lemma}[theorem]{Lemma}
\newtheorem{lemma*}{Lemma}
\newtheorem{defn}[theorem]{Definition}
\newtheorem{remark}[theorem]{Remark}
\newtheorem*{rem}{Remark}
\newtheorem*{probaux}{Problem \protect\probnumber}
\newenvironment{prob}[1]{\def\probnumber{#1}\probaux \mbox{}\\}{\endprobaux}

\newenvironment{letter}{\begin{enumerate}[a)]}{\end{enumerate}}

\newlist{num}{enumerate}{1}
\setlist[num, 1]{label = (\alph*)}
\newcommand{\myitem}{\item}

\newcommand{\Prob}{\mathbb{P}}
\newcommand{\E}{\mathbb{E}}
\newcommand{\Q}{\mathbb{Q}}
\newcommand{\R}{\mathbb{R}}
\newcommand{\N}{\mathbb{N}}
\newcommand{\Z}{\mathbb{Z}}
\newcommand{\F}{\mathcal{F}}
\newcommand{\C}{\mathbb{C}}
\newcommand{\Cl}{\operatorname{Cl}}
\newcommand{\disc}{\operatorname{disc}}
\newcommand{\Gal}{\operatorname{Gal}}
\newcommand{\Img}{\operatorname{Im}}
\newcommand{\eps}{\varepsilon}
\newcommand{\ind}{\mathbbm{1}}

\newcommand\ang[1]{\left\langle#1\right\rangle}

\renewcommand{\L}{\mathbf{L}}

\renewcommand{\d}{\partial}
\newcommand*\Lapl{\mathop{}\!\mathbin\bigtriangleup}
\newcommand{\mat}[4] {\left(\begin{array}{cccc}#1 & #2\\ #3 & #4 \end{array}\right)}
\newcommand{\mattwo}[2] {\left(\begin{array}{cc}#1\\ #2 \end{array}\right)}
\newcommand{\pphi}{\varphi}
\newcommand{\eqmod}[1]{\overset{\bmod #1}{\equiv}}
\renewcommand{\O}{\mathcal{O}}
\renewcommand{\subset}{\subseteq}
\newcommand{\fr}[1]{\mathfrak{#1}}
\newcommand{\p}{\mathfrak{p}}
\newcommand{\m}{\mathfrak{m}}
\newcommand{\supp}{\mathrm{supp}}
\newcommand{\limto}[1]{\xrightarrow[#1]{}}

\newcommand{\under}[2]{\mathrel{\mathop{#2}\limits_{#1}}}

\newcommand{\underwithbrace}[2]{  \makebox[0pt][l]{$\smash{\underbrace{\phantom{%
    \begin{matrix}#2\end{matrix}}}_{\text{$#1$}}}$}#2}

\title{Zeros of Optimal Functions in the Cohn-Elkies Linear Program}
\maketitle

\begin{abstract}
In a recent breakthrough, Viazovska and Cohn, Kumar, Miller, Radchenko, Viazovska solved the sphere packing problem in $\mathbb{R}^8$ and $\mathbb{R}^{24}$, respectively, by exhibiting explicit optimal functions, arising from the theory of weakly modular forms, for the Cohn-Elkies linear program in those dimensions. These functions have roots exactly at the lengths of points of the corresponding optimal lattices: $\{\sqrt{2n}\}_{n\geq 1}$ for the $E_8$ lattice, and $\{\sqrt{2n}\}_{n\geq 2}$, for the Leech lattice. The constructions of these optimal functions are in part motivated by the locations of the zeros. But what are the roots of optimal functions in other dimensions? We prove a number of theorems about the location of the zeros of optimal functions in arbitrary dimensions. In particular, we prove that distances between root lengths are bounded from above for $n \geq 1$ and not bounded from below for $n \geq 2$, and that the root lengths have to be arbitrarily close for arbitrarily long, that is, for any $C, \eps > 0$, there is an interval of length $C$ on which the root lengths are at most $\eps$ apart. We also establish a technique that allows one to improve a non-optimal function in some cases. 
\end{abstract}

\section{Introduction}
The sphere packing problem asks for the maximal fraction $\Delta_n$ of $\R^n$ that can be occupied by open unit balls. While this natural geometrical question has interested mathematicians for at least half a millennium, the exact values of $\Delta_n$ are only known for $n = 1$, $2$, $3$ (\cite{kepler}), and, as of recently, $8$ and $24$ (\cite{dim8}, \cite{dim24}). The upper bounds on $\Delta_8$ and $\Delta_{24}$ rely on a theorem of Cohn and Elkies:
\begin{theorem}[\cite{CEtheorem}, Theorem $3.2$]\label{CE}
Let $r > 0$, and let $f: \R^n \to \R$ be an even function satisfying the following conditions:
\begin{enumerate}
\item $f(0) = \hat{f}(0) > 0$;
\item $f(x) \leq  0$ for $\norm{x} \geq r$;
\item $\hat{f}(t) \geq 0$ for all $t$;
\item $f, \hat{f} \in L^1(\R^n).$ \footnote{Originally, the theorem included a stronger decay at infinity condition, which was removed in Section $9$ of \cite{CohnKumar}. Technically, as stated in \cite{CohnKumar} or \cite[Theorem $3.3$]{CohnZhao}, the theorem applies to an even broader range of functions than as stated above; however, the arguments of the paper carry over verbatim for that larger collection of functions.}
\end{enumerate}
Then:
$$\Delta_n \leq \frac{\pi^{n/2}}{\Gamma(n/2 + 1)} \left(\frac{r}{2}\right)^n.$$ 
\end{theorem}
Since the conditions of the theorem are invariant under rotation about the origin, by replacing $f$ with the average of its rotations, we can assume $f(x)$ is radial (that is, only dependent on $\norm{x}$).

The question arises: what is the smallest $r$ for which such a function $f:\R^n \to \R$ exists? In $2001$, Cohn and Elkies conjectured that this linear program yields sharp upper bounds in dimensions $n = 2, 8\text{ and }24$ (\cite{CEtheorem}). They also produced explicit functions $f$ that gave upper bounds on $\Delta_n$ that differed from the best-known lower bounds by less than $10^{-10}$ for these values of $n$. However, it was not until $2016$ that Viazovska (\cite{dim8}) produced a function that exactly matched the best-known lower bound in $\R^8$ (which arises from the packing where the sphere centers are placed at the points of the $E_8$ lattice). A function in $\R^{24}$ matching the lower bound of the Leech lattice was found shortly thereafter by Cohn, Kumar, Miller, Radchenko and Viazovska (\cite{dim24}), and the case $n = 2$ remains unresolved. 

Theorem \ref{CE} does not give any insight for how to search for functions with the minimal $r$, and the constructions in \cite{dim8, dim24} do not seem to carry over to other dimensions. The proof of Theorem \ref{CE} implies that if an optimal function were to match the lower bounds of the $E_8$ (or Leech) lattice packing, it would have to evaluate to $0$ for $\norm{x}$ equal to vector lengths of that lattice  -- $\{\sqrt{2n}\}_{n\geq 1}$ for the $E_8$ lattice in $\R^8$ and $\{\sqrt{2n}\}_{n\geq 2}$ for the Leech lattice in $\R^{24}$. This observation helped search for the corresponding optimal functions, which have forced roots at those lengths because of a $\sin(\pi \norm{x}^2/2)^2$ factor. In this paper, we look at the properties of zero sets of optimal functions for a general $n$. 

We call a function $f: \R^n \to \R$ \emph{acceptable} if $f$ is even and $f, \hat{f} \in L^1(\R^n)$. For $f$ acceptable, $f, \hat{f}$ must be continuous. 
We normalize the Fourier transform $\hat{f}$ of $f$ by $$\hat{f}(\xi) := \int_{\R^n} f(x) e^{-2 \pi i \xi \cdot x} dx.$$
We let $\F_n$ be the set of acceptable functions $f: \R^n \to \R$ such that $\hat{f}(0) = f(0) > 0, \hat{f} \geq 0$, and $f(x) \leq 0$ when $\norm{x} \geq 0$ for some $r >0$. For $f \in \F_n$, we let 
$$r(f) := \inf \{r \in \R : f(x) \leq 0 \text{ for } \norm{x} \geq r\}, \ \ R(n):= \inf_{f \in \F_n} r(f).$$
We say $f:\R^n \to \R$ is a \textbf{{Cohn-Elkies function}} if $r(f) = R(n)$. A Cohn-Elkies function is not known to exist for every value of $n$, but is believed to, and does not have to be unique for a given $n$.

For sets $A, B \subset \R^n$ and a point $p \in \R^n$, we define $A \pm B : = \{a   \pm b\  \vert \ a, b \in S\},$ $p \pm A := \{p \pm a | a \in A\}$, and $-A := \{-a | a \in A\}$, where points are summed as vectors. We let $d(A, B):= \underset{a \in A, b \in B}{\inf} d(a, b),$ where $d(a, b)= \norm{a-b}$ is the Euclidean distance between $a$ and $b$.  We let $N_\eps (A) := \{p : d(p, A) < \eps\}$ be the $\eps$-neighborhood of the set $A$.  For a set $A$, we let $\overline{A}$ be its closure, and if $A$ is Lebesgue measurable, we use $\mu(A)$ to denote its Lebesgue measure. We let $\R_+:= \{r \in \R: r > 0\}$.

For $p \in \R^n$, $r > 0$, we let $B^n_r(p) \subseteq \R^n$ be the closed $n$-ball of radius $r$ centered at $p$. By the unit ball in $\R^n$ we mean $B_1^n(0)$. We let $V_n := \frac{\pi^{(n/2)}}{\Gamma(n/2 + 1)}$ be the volume of the unit ball in $\R^n$. By \emph{length} of an element $z \in \R^n$ we mean its distance from the origin, $\norm{z}$. For a subset $A \subseteq \R^n$, we let $\L({A}): = \{\norm{a} \ \vert \ a \in A\}$ be the set of lengths of elements of $A$. We let $A^c$ denote the complement of $A$. For a function $f: \R^n \to \R$, we call the set $\{z \in \R^n: f(z) = 0\}$ the zero set of $f$.

The paper is structured as follows. First, we prove the main theorem:
\begin{theorem}\label{main}
Let $f: \R^n \to \R$, $f \in \F_n$ be a radial Cohn-Elkies function. Let $r := r(f) = R(n)$. Let $S \subseteq \R^n$ be a compact set of measure $\mu(S) > 1$. Then:
\begin{enumerate}
\item There exist $z \in (S - S) \cap \{x \in \R^n : \norm{x} \geq r\}$ such that $f(z) =  0$.
\item There exist $z' \in S - S$ such that $\hat{f}(z') = 0$.
\end{enumerate}
\end{theorem}
Then we prove some corollaries of Theorem \ref{main}, for example that the zero set is unbounded, the distances between lengths of zeros are bounded from above and not bounded from below. \\

\section{Main result and its corollaries}

\begin{proof}[Proof of Theorem \ref{main}]
Let $Z$ denote the zero set of $f$, and suppose for contradiction that $S - S$ contains no $z \in Z$ with $\norm{z} \geq r$. 
Let 
$$h(x) :=  \ind_S* \ind_{-S}(x) = \int_{\R^n} \ind_S(t)  \ind_{-S}(x-t) dt = \int_{S} \ind_{S}(t - x)  dt = \mu(S \cap (S + x))  .$$ 
Since $S$ is compact, $\abs{h}$ is bounded from above. Moreover, the set $T:= (S - S) \cap \{x \ \vert \ \norm{x} \geq r\}$ is also compact, and by assumption $f$ is strictly negative on $T$, and thus there exists some $c < 0$ such that $f(x) < c$ for all $x \in T$. Hence, we can find some $\alpha > 0$ such that 
$$F(x) := f(x) + \alpha h(x) \leq 0$$ for all $x \in \R^n$ with $\norm{x} \geq r$. 

Now, $h$ has compact support $\supp(h) \subset \overline{S - S} = S-S,$ and $\abs{h}$ is bounded, so of course $h \in L^1(\R^n)$. Moreover, $\hat{h} = |\widehat{\ind_S}|^2 \in L^1(\R^n)$ as well, because $\ind_S \in L^2(\R^n)$ and Fourier transform preserves the $L^2$ norm. Hence, both $h$ and $\hat{h}$ are acceptable.

Next, note  $$\widehat{F} = \hat{f} + \alpha \hat{h} =  \hat{f} + \alpha \abs{\widehat{\ind_S}}^2 \geq 0$$ for all $x \in \R^n$. 

Lastly, note that 
$$h(0) = \int_{\R^n} \ind_S(t) \ind_{-S}(-t) dt = \mu(S) <\mu(S)^2 = \left(\int_{\R^n}  \ind_S(t) dt \right)^2 = \hat{h}(0),$$ so $$F(0) = f(0) + \alpha h(0) < \hat{f}(0) + \alpha \hat{h}(0) = \widehat{F}(0).$$ Let $c > 1$ be such that $F(0) = \widehat{F}(0)/c^n$, and let $F_c(x):= F(cx)$. Then:
\begin{itemize}
\item $\widehat{F_c}(0) = \widehat{F}(0)/c^n = F(0) = F_c(0)$;
\item $ \widehat{F_c}(x) = \widehat{F}(x/c)/c^n \geq 0$ for all $x$ since $\widehat{F} \geq 0$;
\item $F_c(x) \leq 0$ for all $x$ with $\norm{x} \geq r/c$;
\item $F_c$ is acceptable because $F$ is acceptable as a linear combination of acceptable functions. 
\end{itemize} 
Thus, we have found a function $F_c \in \F_n$ with $r(F_c) = r/c < r$. This is a contradiction.
\\

The proof of the second statement is almost identical. Suppose for contradiction that $S - S$ contains no zero of $\hat{f}$. Define the function $h$ in the same way as above. Since by hypothesis $\hat{f}$ is strictly positive on the support of $h$, we can find $\alpha > 0$ such that $\hat{f}(x) >  \alpha h(x)$ for all $x$. Let $$F(x):= f(x) - \alpha \hat{h}(x) = f(x) - \alpha \abs{\widehat{\ind_S}}^2 \leq f(x),$$ so $F(x) \leq 0$ for $\norm{x} \geq r$. Then $\widehat{F}(x) = \hat{f}(x) - \alpha h(x) \geq 0$ everywhere. Moreover, we again have 
$$F(0) = f(0) - \alpha \hat{h}(0) < \hat{f}(0) - \alpha\abs{h(0)}= \widehat{F}(0) .$$ Repeating the last step of the proof of the first part, we again arrive to a contradiction.
\end{proof}
The proof of Theorem \ref{main} gives us a technique to improve functions $f \in \F_n$ with $r(f) > R(n)$ in certain cases: 

\begin{corollary}
Let $f \in \F_n$ and suppose there exists $S \subseteq \R^n$ with $\mu(S) > 1$ such that $S-S$ contains no zeros of $f$. Then there exists $\alpha > 0$ and $c > 1$ such that the function $$F_c(x) := f(cx) + \alpha \cdot \ind_S* \ind_{-S}(cx)$$ satisfies $r(F_c) < r(f).$ 

Similarly, let $f \in \F_n$ and suppose there exists $S \subseteq \R^n$ with $\mu(S) > 1$ such that $S-S$ contains no zeros of $\hat{f}$. Then there exists $\alpha > 0$ and $c > 1$ such that the function $$F_c(x) := f(cx) - \alpha \cdot \abs{\widehat{\ind_S} (cx)}^2$$ satisfies $r(F_c) < r(f).$ 

In both cases, the parameters $\alpha, c$ can be chosen as in the proof of Theorem \ref{main}.
\end{corollary}

Now we look at what happens when $S$ is a union of disjoint balls. The next result is not not new -- we state it as the simplest case of Theorem \ref{main}.
\begin{corollary} 
 For $n \in \N$, suppose a Cohn-Elkies function $f : \R^n \to \R$ exists. Then:  $$R(n) \leq  2V_n^{(-1/n)} =  2 \left(\frac{\Gamma(n/2 + 1)}{\pi^{n/2}}\right)^{1/n} = (1 + o(1)) \sqrt{\frac{2n}{\pi e }} .$$
\end{corollary}
\begin{proof}
Let $f:\R^n \to \R $ be a Cohn-Elkies function, let $r:=V_n^{(-1/n)}$, and let $\eps > 0$. Note that $\mu(B^n_{r + \eps}) > 1.$ Thus Theorem \ref{main} applies, and so the set $$B^n_{r + \eps}(0) - B^n_{r\ + \eps}(0) = B^n_{2r + 2\eps}(0)$$ contains a zero $z$ of $f$ with $\norm{z} \geq R(n)$. Hence, $R(n) \leq \norm{z} \leq 2r + 2\eps$. Since this is true for any $\eps > 0$, it follows $$R(n) \leq 2r =  2V_n^{(-1/n)}.$$ The asymptotic expansion follows from Stirling's approximation.
\end{proof}
In all the known Cohn-Elkies functions (in dimensions $n = 1, 8, 24$) the distances between root lengths go to zero. In the rest of this section we prove results about distances between root lengths for an arbitrary dimension $n$. Recall that for a set $A \subset \R^n$, $\L(A) := \{\norm{a} \}_{a \in A} \subseteq \R_+.$
\begin{theorem}
Let $n \in \N$, let $f: \R^n \to \R$ be a radial Cohn-Elkies function, and let $Z \subseteq \R^n$ be the set of roots of $f$. Then there exists $C > 0$ such that for any $t > 0$,  $[t, t + C] \cap \L(Z) \neq \emptyset$. Similarly, let $Z'$ be the zero set of $\hat{f}$. Then there exists $C' > 0$ such that for all $t > 0$, $[t, t + C'] \cap \L(Z') \neq \emptyset.$ 
\end{theorem}
\begin{proof}
Suppose for contradiction this is not the case. Let $\eps$ be such that $0 < 2\eps < \inf \L(Z)$ (recall that $f$ has to be continuous so $\inf \L(Z) > 0$). 
We will show by induction that for any $k \in \N$, there exists a set $S_k = \bigcup_{i = 1}^k b_i$ of $k$ disjoint balls of radius $\eps$ such that $S_k - S_k \subseteq Z^c$.  

\begin{itemize}
\item For $k = 1$, let $b_1:= B^n_\eps(0)$. Then: $S_1 -S_1 = B^n_{2\eps}(0) \subseteq Z^c$ by choice of $\eps$. 

\item Suppose we have constructed $S_k$. Let $D$ be the diameter of $S_k$. By assumption, we can find $t > 0$ such that $[t, t + D + 2 \eps]$ contains no elements of $\L(Z)$. Choose a point $p$ such that $d(p, S_k) = t + \eps$ and let $b_{k + 1} = B^n_\eps(p)$. For all $q \in S_k$, $d(p, q) \in [t + \eps, t + D + \eps]$, so $d(S_k, b_{k + 1}) \subseteq [t, t + D + 2\eps]$. Hence, $S_{k + 1} := S_k \cup b_{k + 1}$ satisfies $S_{k + 1} - S_{k + 1} \subseteq Z^c  $ as desired.
\end{itemize}
For sufficiently large $k$, $\mu(S_k) > 1$. This contradicts Theorem \ref{main}. The second one is proven verbatim by replacing $f$ with $\hat{f}$ and $Z$ with $Z'$.
\end{proof}

\begin{corollary}
Let $f: \R^n \to \R$ be a Cohn-Elkies function, and let $Z, Z'$ denote the sets of roots of $f$ and $\hat{f}$. Then $Z, Z'$ are unbounded. 
\end{corollary}

\begin{theorem}
Let $f: \R^n \to \R$ be a radial Cohn-Elkies function and let $Z$ be the zero set of $f$ (or $\hat{f}$). Let $T \subseteq \R^n$ be any unbounded set, and let $\eps > 0$. Then for any $C>0$, $N_\eps(T - T)$ contains infinitely many $z \in Z$ with $\norm{z} > C$.
\end{theorem}
\begin{proof}
Without loss of generality, $\eps < \inf \L(Z)$, so $N_\eps(0)$ contains no elements of $Z$. 

First, note that since $T$ is unbounded, the set $N_\eps(T - T)$ has infinite measure, is unbounded, and contains infinitely many disjoint closed balls of radius $\eps/2$. Selecting a finite collection of them of total measure greater than $1$ and applying Theorem \ref{main} we see that $N_\eps(T - T)$ contains an element of $Z$. 

Since $T$ is unbounded, for any $C >0$ we can find an unbounded $T' \subseteq T$ such that the elements of $T'$ are distance $C + 2\eps$ or more apart. Then all elements of $N_\eps(T' - T') \setminus N_\eps(0)$ are of length at least $C$. From the first observation, $N_\eps(T' - T')$ contains an element $z$, and since $N_\eps(0) \cap Z = \emptyset$, we can conclude $\norm{z} > C$. Thus $N_\eps(T - T)$ contains an element $z$ with $\norm{z} > C$. 

Lastly, there are infinitely many of such elements because we can keep increasing $C$ to be larger than the length of all the elements already chosen. 
\end{proof}
\begin{corollary}
Let $f: \R^n \to \R$ be a radial Cohn-Elkies function and let $Z$ be the zero set of $f$ (or $\hat{f}$). Let $\Lambda \subset \R$ be a lattice. Then $N_\eps(\Lambda)$ contains infinitely elements of $\L(Z)$.
\end{corollary}

In the rest of the section we show that the lengths of zeros of a Cohn-Elkies function have to be arbitrarily close on arbitrarily long intervals when $n \geq 2$ (Theorem \ref{strongestcor}). We begin with a few purely geometric lemmas (\ref{infiniteintervals}, \ref{subinterval}, \ref{geolemma}).

\begin{defn}
For a set $T \subseteq \R_+$ and $\gamma \in \R$, we define $T^\gamma \subset \R^2$ as:
$$T^\gamma := \{(\norm{t}\cos \gamma , \norm{t} \sin \gamma) \vert t \in T\}.$$ We let 
$$\ell^\gamma:= \R_+^\gamma$$ be the ray emitting from $0$ at angle $\gamma$ to the $x$ axis.
\end{defn}

\begin{lemma}\label{infiniteintervals}
Let $I_0 \subseteq \R$ be an interval and let $K$ be an infinite set of intervals contained in $I_0$ such that $\inf_{I \in K} \mu({I}) > 0$. Then there exists an infinite subset $L \subseteq K$ and an interval $J$  with $\mu(J) > 0$ such that $J \subseteq I$ for all $I \in L$. 
\end{lemma}
\begin{proof}
  Let $c := \inf_{I \in K} \mu({I})$. For every $I \in K$, let $p_I := \sup I$ be the rightmost point of $I$. Since $K$ is infinite and $I_0$ is bounded, the set $\{p_I\}_{I \in K}$ has a limit point, so there is a point $p$ and an infinite subset $L \subseteq K$ such that $\norm{p_I - p} < c/4$ for all $I \in L$. Then the interval $J:= [p - 5c/8, p - 3c/8]$ is contained in all $I \in L$. 
\end{proof}
\begin{lemma}\label{subinterval}
Let $A \subseteq \R$ be such that for some $C, c > 0$, for any closed interval $J$ with $\mu(J) = C$, $J \cap A$ contains a closed interval of length $c$. Let $T \subseteq \R$ be unbounded, and let $I$ be an interval, $\mu({I}) = C$. 
Then there exists an interval $I' \subseteq I$ with $\mu(I') >0$ and an unbounded subset $T' \subseteq T$ such that $T' - I'  \subseteq A$.
\end{lemma}
\begin{proof}
For each $t \in T$, the interval $I_t : = t - I$  has length $C$, hence by assumption there exists an interval $J_t \subseteq I$ such that $\mu(J_t) = c$ and $t - J_t  \subseteq A$. 
It remains to note that by Lemma  \ref{infiniteintervals}, there exists an unbounded $T' \subset T$ a positive measure interval $I'$ such that $I' \subset J_t$ for all $t \in T'$, and hence $t - I' \subseteq t - J_t \subseteq A$ for all $t \in T'$.
\end{proof}
\begin{lemma}\label{geolemma}
Let $B \subseteq \R^2 \setminus \{(0, 0)\}$ be a radial set such that the set $A:= \L(B) \cup -\L(B)$ satisfies the hypotheses of Lemma \ref{subinterval}. Then for all $k \in \N$ and $\eps > 0$, there exist points $p_1, \ldots, p_k \in \R^2$, an interval $[\alpha_k, \beta_k]$ and an unbounded set $T_k \subseteq \R_+$ such that:
\begin{enumerate}
\item For all $i \neq j$, $$\norm{p_i - p_j} > \eps;$$ 
\item The set $P_k:= \bigcup_{i = 1}^k p_i$ satisfies $$P_k - P_k \subseteq B^c;$$
\item For all $\gamma \in [\alpha_k, \beta_k]$, $$T_k^\gamma - P_k \subseteq B^c.$$
\end{enumerate}

\end{lemma}
\begin{proof}
Let $A_+ := \L(B)$, so $A:= A_+ \cup -A_+$. Since $A$ satisfies the hypotheses of Lemmma \ref{subinterval}, $A_+$ is unbounded. We prove the statement of the theorem by induction.
\begin{itemize}
\item For $k = 1$, let $p_1:=( 0, 0)$, let $\alpha_1:=0; \beta_1:= \pi/4$, and let $T_1:=A_+$. 
\begin{enumerate}
\item Condition $1$ is trivial.
\item $p_1 - p_1 =(0, 0) \not \in B$ by assumption, so $P_1 - P_1 \subseteq B^c$.
\item Lastly, $T_1^\gamma - P_1 = T_1^\gamma = A_+^\gamma \subseteq B$ since $B$ is radial. 
\end{enumerate}
\item Suppose the statement holds for some $k$. We prove it holds for $k + 1$. Without loss of generality (by rotating all the points about the origin), we have $\alpha_k = 0, \beta_k =:\beta$. 
Let $$p \in T_k^0, \norm{p} > \frac{C}{(1 - \cos \beta)} +  \sum_{i = 1}^k \norm{p_i} + \eps$$ (which exists since $T_k$ is unbounded).

\begin{enumerate}
\item Since  $\norm{p} > \norm{p_i} + \eps$, we have $\norm{p - p_i} > \eps$ for all $i \leq k$. 
\item Since $p \in T_k^0$ and $0 \in [\alpha_k, \beta_k]$, by the third condition of the induction assumptions, $p - P_k \subseteq B^c$, and since $B$ is radial, $P_k - p \subseteq B^c$ as well. Hence, $P_{k + 1}:= p \cup P_k$ satisfies $$P_{k + 1} - P_{k + 1}  \subseteq B^c.$$
\item Consider the interval $I:= [\norm{p} \cos (\beta), \norm{p}] \subset \R_+.$ By choice of $p$, $\mu({I}) > C$, so by Lemma \ref{subinterval}, there exists a positive measure interval $[a, b] = J \subset I$ and an unbounded subset $T \subset T_k$ such that
\begin{align}\label{shmeep}
T - J \subset A.
\end{align}

Now, let $q_\gamma$ be the projection of $p$ onto $\ell^\gamma$. We choose $\alpha_{k + 1}, \beta_{k + 1}$ to satisfy: 
$$\norm{p} \cos(\beta_{k +1}) = a + \mu(J)/3, \norm{p} \cos(\alpha_{k + 1}) = b -  \mu(J)/3 = a + 2 \mu(J)/3.$$ Note that $[\alpha_{k +1}, \beta_{k + 1}] \subseteq [0, \beta]$ and for all $\gamma \in [\alpha_{k  +1}, \beta_{k + 1}]$, 
\begin{align}\label{lamma}
\norm{q_\gamma} = \norm{p} \cos(\gamma) \subseteq [a + \mu(J)/3, b - \mu(J)/3].
\end{align} 

Let $\gamma \in [\alpha_{k +1}, \beta_{k + 1}].$ Note that for $x \in \ell^\gamma$, we have $d(x, p) - d(x, q_\gamma) \limto{\norm{x} \to \infty} 0$ uniformly for all $\gamma$. Hence, for sufficiently large $M > 0,$ the set $$T_{k + 1} := \{x \in T \ \vert  \norm{x} > M\}$$ satisfies:
\begin{align}\label{meep}
\abs{d(x, p) - d(x, q_\gamma)}  < \mu(J)/3 \text{ for all }x \in T_{k + 1}^\gamma.
\end{align}
We choose $M$ sufficiently large to also satisfy $M > \norm{p}$ (and hence $M > \norm{q_\gamma}$ for all $\gamma$). 

We now collect all of the above together. 
By the induction hypothesis, we already know that 
$$T^\gamma_{k + 1} - P_k \subseteq T_k^\gamma - P_k \subseteq B.$$ It remains to show that 
$$T_{k + 1}^\gamma - p \subseteq B.$$ Indeed, let $t \in T^\gamma_{k + 1}$. Since $t, q_\gamma$ and the origin lie on a line and since $\norm{t} \geq M > \norm{q_\gamma}$, we have 
$$\norm{t  - q_\gamma} = \norm{t} - \norm{q_\gamma}.$$
By (\ref{lamma}), 
$$ \norm{t} - \norm{q_\gamma} \in [\norm{t} - b + \mu(J)/3, \norm{t} - a - \mu(J)/3].$$
Lastly, by (\ref{meep}) it follows that
$$\norm{t - p} \in [\norm{t  - q_\gamma} - \mu(J)/3, \norm{t  - q_\gamma} +\mu(J)/3] \subseteq \norm{t} - J.$$ Since $\norm{t} \in T_{k + 1} \subseteq T$, by (\ref{shmeep}) we see that 
$$\norm{t - p} \subseteq A,$$ which by radiality of $B$ implies that $$t - p \in B.$$
 This completes the inductive step.

\end{enumerate}
\end{itemize}
\end{proof}

\begin{theorem}\label{strongestcor}
Let $n \geq 2$, let $f: \R^n \to \R$ be a radial Cohn-Elkies function, and let $Z$ be the set of zeros of $f$ (or $\hat{f}$). Then for any $C, c > 0$, there exists an interval $I$ of length $C$ such that any subinterval $J \subseteq I$ of length $c$ contains an element of $\L(Z)$.
\end{theorem}
\begin{proof}
We again prove this by contradiction. Suppose this is not the case. Then there exist $C, \tilde c > 0$ such that any interval $I$ with $\mu(I) = C$ contains a subinterval of length $\tilde c$ that avoids $\L(Z)$. Let $\eps$ be such that $$0 < \eps < \min\{\tilde c/2,\inf \L(Z)/2, C/4\},$$
 and let $$c:= \tilde c - 2\eps>0.$$ 
Then any interval of length $C$ contains a subinterval of length $c$ that avoids $\L(N_\eps(Z)) = N_\eps(\L(Z))$. Let $B \subseteq \R^2$ be a radial set with $\L(B) = [\L(N_\eps(Z))]^c$. By choice of $\eps$, $(0, 0) \not \in B$. We claim that the set $B$ satisfies the hypotheses of Lemma \ref{geolemma}. Indeed, consider any interval $I \subseteq \R$ of length $C$. If $I$ is contained in $\R_{\geq 0}$ or $\R_{\leq 0}$, this follows directly from the hypotheses. Otherwise $(0, 0) \in I$, so since $c < \inf  \L(Z)/2$ and $c < C/2$, we see that $I$ contains one of the two intervals $[0, \inf \L(Z)]$ or $[-\L(Z), 0]$, and hence again contains a subinterval of length $c$ that avoids $\L(B) \cup -\L(B).$ 

Applying Lemma \ref{geolemma}, we get that for any $k$, we can find $k$ distinct points $p_1, \ldots, p_k$ in $\R^2$ such that $\norm{p_i - p_j} > \eps$ for $i \neq j$ and $\norm{p_i - p_j} \in \L(N_\eps(Z))^c$ for all $i, j$. Choose any distance-preserving embedding of these points into $\R^n$, and let $q_i$ be the image of $p_i$ under this embedding. Note that the set $Q:= \bigcup q_i$ satisfies $$Q - Q \in (N_\eps(Z))^c,$$ and hence 
$$N_{\eps/2}(Q) - N_{\eps/2}(Z) =  N_\eps(Q - Q) \subseteq Z^c.$$ Let $S:= N_{\eps/2}(Q)$. Note that $Q$ is a disjoint union of $k$ balls. For sufficiently large $k$, $\mu(S) >1$. This contradicts Theorem \ref{main}. 
\end{proof}

\begin{corollary}
Let $n \geq 2$, let $f: \R^n \to \R$ be a radial Cohn-Elkies function, and let $Z$ be the set of zeros of $f$ (or $\hat{f}$). Then for any $\eps >0$ and $K > 0$, there exists $x > K$ such that $[x,  x + \eps] \cap \L(z) $ contains at least $2$ points. 
\end{corollary}

\begin{rem}
Calculations performed by Henry Cohn suggest that for $n \geq 2$, the $\limsup$ of the distances between lengths of zeros of a Cohn-Elkies function $f:\R^n \to \R$ is $0$ as well. We note that this cannot be shown using Theorem \ref{main}. For example, for a radial set $Z$ with 
$$\L(Z^c) = [0, \eps] \cup \{[2^k - \eps, 2^k + \eps]\}_{k \in \N},$$ for a sufficiently small $\eps$, there is no contradiction with Theorem \ref{main} with regards to $Z$ being a zero set of $f$. 
\end{rem}
\section{Simultaneous roots of $f$ and $\hat{f}$}
Among known Cohn-Elkies functions, $f$ and $\hat{f}$ have a lot of common roots. For example, in $8$ dimensions, the known radial optimal function $f$ and its Fourier transform $\hat{f}$ have the exact same set of root lengths $\{\sqrt{2n}\}_{n\geq 1}$, and in $24$ dimensions both $f$ and $\hat{f}$ have root lengths $\{\sqrt{2n}\}_{n\geq 2}$. It would be of interest to prove a general result about common roots of $f$ and $\hat{f}$. 
\begin{theorem}
Let $f: \R^n \to \R$ be a radial Cohn-Elkies function and let $r:= r(f)$ and assume that $\hat{f}(x)$ does not vanish for $\norm{x} \in [0; r)$. Then: $\hat{f}(x) = 0$ for $\norm{x} = r$. 
\end{theorem}
The assumption that $\hat{f}$ does not vanish on $[0, r)$ is consistent with the known Cohn-Elkies functions; moreover, $f$ also appears not to vanish in this region. With these assumptions, the theorem asserts that $f$ and $\hat{f}$ share their smallest root. 
\begin{proof}
Suppose for contradiction that $\hat{f}$ does not vanish in the region $\norm{x} \leq r$ (so $\hat{f}$ is strictly positive in this region). Consider the function $F: = f - c \hat{f}$, where $0 < c < 1$. Then: 
\begin{itemize}
\item $F$ is acceptable;
\item $F(0) = f(0) - c \hat{f}(0) = \hat{f}(0) - c f(0) = \widehat{F}(0) > 0$; 
\item $\widehat{F}(x) = \hat{f}(x) - c f(x) \geq 0$ for $\norm{x} \geq r$; 
\item $F(x) \leq 0$ for $\norm{x} \geq r$, and $F(x) = - c \hat{f}(x) < 0$ for $\norm{x} = r$, which means that there is some $r' < r$ such that $F(x) \leq 0$ for $\norm{x} > r'$. 
\end{itemize}
Since $\hat{f}(x) > 0$ for $\norm{x} \leq r$, $\hat{f}$ is bounded below by some constant $C> 0$ in that region, which means we can choose $c$ sufficiently small to have $\widehat{F} \geq 0$ everywhere. We have found $F \in \F_n$ with $r(F) = r' < r$, leading to a contradiction. 
\end{proof}
\begin{corollary}
Let $f: \R^n \to \R$ be a radial Cohn-Elkies function. Then $\hat{f}$ has a root in the region $\norm{x} \leq r(f)$.

\end{corollary}
\section{Acknowledgments}
 I would like to thank Henry Cohn for being a superb mentor for this project, as well as Microsoft Reseach New England for providing a great working environment over the summer of $2018$. 

\end{document}